\documentclass[runningheads,a4paper]{llncs}

\usepackage{cite}
\usepackage[cmex10]{amsmath}
\usepackage{amssymb}
\usepackage{url}
\usepackage[all]{xy}
\usepackage{latexsym}
\usepackage{xspace}

\input diagxy

\newcommand{\word}[1]  {
  \mathit{#1}
}

\newcommand{\lift}[2]  {
  {\word{lift}_{#1} \; {#2}}
}

\newcommand{\oplift}[2]  {
  {\word{oplift}_{#1} \; {#2}}
}

\newcommand{\catl}[1]  {
  \mathbb{#1}
}
\newcommand{\catw}[1]  {
  \mathbf{#1}
}

\newcommand{\arr}[1]  {
  {#1}^{\rightarrow}
}

\newcommand{\arrl}[1]  {
  \arr{\catl{#1}}
}
\newcommand{\mor}[3]  {
  #1 \colon #2 \rightarrow #3
}

\newcommand{\id}[1]  {
	{\mathit{id}_{#1}}
}

\begin{document}

\title{The Other Pullback Lemma
}


\author{Michal R. Przybylek    
}

\institute{Faculty of Mathematics, Informatics and Mechanics\\
University of Warsaw\\
Poland}


\maketitle

\begin{abstract}
Given a composition of two commutative squares, one well-known pullback lemma says that if both squares are pullbacks, then their composition is also a pullback; another well-known pullback lemma says that if the composed square is a pullback and the supporting square is a pullback, then the other square is a pullback as well. This note  investigates the third of the possibilities.
\end{abstract}

\section{Introduction}
\label{intro}
Consider a commutative diagram:
$$\bfig
\place(200,150)[(I)]
\place(600,150)[(II)]
	\node a(0, 300)[A]
	\node b(400, 300)[B]
	\node c(800, 300)[C]
	
	\node x(0, 0)[X]
	\node y(400, 0)[Y]
	\node z(800, 0)[Z]

	\arrow[a`b;]
	\arrow[b`c;]
	
	\arrow[x`y;]
	\arrow[y`z;]

	\arrow[a`x;]
	\arrow[b`y;]
	\arrow[c`z;]
\efig$$
There are two well-known pullback lemmata\cite{maclane}:
\begin{enumerate}
\item if squares (I) and (II) are pullbacks, then the outer square is a pullback (i.e. pullbacks compose)
\item if the outer square and the right square (II) are pullbacks, then square (I) is a pullback
\end{enumerate}
The third possibility: ``if the outer square and the left square (I) are pullbacks, then the right square (II) is a pullback'', which we shall call ,,the other pullback lemma'', generally does not hold. 
More interestingly, the other pullback lemma does not have to hold even in case the left bottom morphism is an epimorphism.
\begin{example}[Counterexample to the other pullback lemma]
Let $\catw{Graph}$ be the category of simple graphs and graph homomorphisms. Consider an injective-on-nodes homomorphism ${\mor{e}{|2|}{2}}$ from the discrete graph $|2|$ consisting of two nodes, to the graph $2$ consisting of a single edge (between two distinct nodes). Homomorphism $e$ is clearly an epimorphism, but in the following situation:
$$\bfig
	\node a(0, 300)[|2|]
	\node b(400, 300)[|2|]
	\node c(800, 300)[2]
	
	\node x(0, 0)[|2|]
	\node y(400, 0)[2]
	\node z(800, 0)[2]

	\arrow[a`b;\id{|2|}]
	\arrow[b`c;e]
	
	\arrow[x`y;e]
	\arrow[y`z;\id{2}]

	\arrow[a`x;\id{|2|}]
	\arrow[b`y;e]
	\arrow[c`z;\id{2}]
\efig$$
the right square is not a pullback square, but both the left an the outer ones are pullback diagrams.
\end{example} 
Even more interestingly, the other pullback lemma does not have to hold in case the left bottom morphism is an equationally-defined epimorphism\footnote{Such epimorphisms are called regular epimorphisms.}.
\begin{example}[Another counterexample to the other pullback lemma]\label{e:another}
Let $\catw{Pos}$ be the category of partially-ordered sets and monotonic functions. Denote by $n$ the initial segment of natural numbers consisting of $n$ elements ${\{0 \leq \cdots \leq n-1\}}$. Consider functions $\mor{f, g}{1}{2 \sqcup 2}$ defined as $f(0) = 1$ and $g(0) = 0'$, where $2 \sqcup 2 = \{0 \leq 1,0' \leq 1' \}$. The coequalisator $3$ of $f$ an $g$ gives an epimorphism $\mor{e}{2 \sqcup 2}{3}$, which identifies $1$ with $0'$: i.e.~${e = \{0 \mapsto 0, 1\mapsto 1, 0' \mapsto 1, 1' \mapsto 2\}}$. The other pullback lemma is does not hold over $e$. In the diagram:
$$\bfig
	\node a(0, 400)[|2|]
	\node b(800, 400)[|2|]
	\node c(1600, 400)[2]
	
	\node x(0, 0)[2 \sqcup 2]
	\node y(800, 0)[3]
	\node z(1600, 0)[3]

	\arrow[a`b;\id{|2|}]
	\arrow[b`c;\{0 \mapsto 0, 1 \mapsto 1\}]
	
	\arrow|b|[x`y;e]
	\arrow|b|[y`z;\id{3}]

	\arrow[a`x;\id{|2|}]
	\arrow|m|[b`y;\{0 \mapsto 0, 1 \mapsto 2\}]
	\arrow|m|[c`z;\{0 \mapsto 0, 1 \mapsto 2\}]
\efig$$
where $|2|$ is the discrete poset on two elements $\{0, 1\}$, the right square is not a pullback square, but both the left square and the outer square are pullbacks.
\end{example}

The aim of this note is to investigate the other pullback lemma. Here is the formal definition.
\begin{definition}[The other pullback lemma]\label{d:the:other:pullback}
Let $\catl{B}$ be a category and ${\mor{e}{X}{Y}}$ a morphism in $\catl{B}$ such that all pullbacks along $e$ exist. We say that $e$ satisfies the other pullback lemma (or that the other pullback lemma is satisfied along ${\mor{e}{X}{Y}}$) if for any commutative squares:
$$\bfig
\place(200,150)[(I)]
\place(600,150)[(II)]
	\node a(0, 300)[A]
	\node b(400, 300)[B]
	\node c(800, 300)[C]
	
	\node x(0, 0)[X]
	\node y(400, 0)[Y]
	\node z(800, 0)[Z]

	\arrow[a`b;]
	\arrow[b`c;]
	
	\arrow[x`y;e]
	\arrow[y`z;]

	\arrow[a`x;]
	\arrow[b`y;]
	\arrow[c`z;]
\efig$$
the following holds: if the outer square and the left square (I) are pullbacks, then the right square (II) is a pullback. 
\end{definition}
Prompted by MathOverflow question\footnote{Ma Ming, \emph{Two pullback diagram}, \url{http://mathoverflow.net/questions/143070/two-pullback-diagram}.} and expanding the comment provided by the author of this note, we provide sufficient and necessary conditions on morphism $\mor{e}{X}{Y}$ such that the other pullback lemma along $\mor{e}{X}{Y}$ holds. 

The note is structured as follows. Section \ref{sec:preliminaries} recalls various notions of epimorphisms (i.e.~extremal epimorphism, strong epimorphism, regular epimorphism) they basic properties and relations that hold between them. In the section we also state our main theorem (Theorem \ref{t:the:other:pullback:lemma}) and discusses some of its corollaries. However, we postpone the proof of the theorem until Section \ref{sec:proof}. The aim of Section \ref{sec:characterisation} is to provide an alternative, more natural, setting for pullback lemmata. Section \ref{sec:proof} gives a proof of a generalisation of our main theorem in the setting provided by Section \ref{sec:characterisation}. 

\section{Preliminaries}
\label{sec:preliminaries}
Let us recall \cite{borceux} that a morphism $\mor{e}{X}{Y}$ is called \emph{strong} (left orthogonal) if for every commutative square:
$$\bfig
	\node a(0, 500)[X]
	\node b(500, 500)[Y]
	
	\node c(0, 0)[A\;]
	\node d(500, 0)[B]

	\arrow[a`b;e]
	\arrow/-->/[b`c;d]
	
	\arrow/>->/[c`d;m]
	\arrow[a`c;f]
	\arrow[b`d;g]
\efig$$
with a monomorphism $\mor{m}{A}{B}$, there exists a (necessarily unique) diagonal $\mor{d}{Y}{A}$ such that $d \circ e = f$ and $m \circ d = g$. If $\mor{e}{X}{Y}$ is strong then it factros through no proper subobject: i.e.~if $e = m \circ f$, where $m$ is a monomorphism, then $m$ is an isomorphism. This is due to the fact, that the diagonal of the square induced by the decomposition $e = m \circ f$  makes monomorphism $m$ a split epimorphism. The converse is not generally true. However, it is true on condition the category has enough pullbacks. For let us assume that we are given a diagram like in the above, where $\mor{e}{X}{Y}$ factors through no proper subobjects, and our task is to construct a diagonal $\mor{d}{Y}{A}$. Let us build the pullback of $m$ along $g$:
$$\bfig
	\node a(-300, 600)[X]

	\node p(0, 300)[m^*(Y)\;]
	\node b(400, 300)[Y]
	
	\node c(0, 0)[A\;]
	\node d(400, 0)[B]

	\arrow/{@{>}@/^1em/}/[a`b;e]
	\arrow|l|/{@{>}@/_1em/}/[a`c;f]
	\arrow/-->/[a`p;!\exists h]

	\arrow/>->/[p`b;\pi_Y]
	
	\arrow/>->/[c`d;m]
	\arrow[p`c;\pi_A]
	\arrow[b`d;g]
\efig$$
Because monomorphisms are stable under pullbacks, $\pi_Y$ is also a monomorphism. Therefore, it is an isomorphism, since $e$ is extremal. One may then construct the diagonal $\mor{d}{Y}{A}$ as the composition $\pi_A \circ \pi_Y^{-1}$. The lower triangle commutes because: $m \circ \pi_A \circ \pi_Y^{-1} = g \circ \pi_Y \circ \pi_Y^{-1} = g$. To see that the upper triangle commutes, let us precompose $\pi_A \circ \pi_Y^{-1} \circ e$ with $m$; then $m \circ \pi_A \circ \pi_Y^{-1} \circ e = g \circ e = m \circ f$, and since $m$ is mono: $\pi_A \circ \pi_Y^{-1} \circ e = f$.

Morphisms that factors through no proper subobjects are called \emph{extremal morphisms}. Let us also recall that if a category has equalisers, every extremal morphism (therefore, every strong morphism) $\mor{e}{X}{Y}$ is an epimorphism because of the following observation. For every pair of morphisms $\mor{a,b}{Y}{Z}$ one may form an equaliser:
$$\bfig
	\node x(300, 300)[X]
	\node a(600, 0)[Y]
	\node b(1000, 0)[Z]
	\node e(50, 0)[\mathit{Eq}(a,b)\;]
	
	\arrow/{@{>}@/^0.5em/}/[a`b;a]
	\arrow|b|/{@{>}@/_0.5em/}/[a`b;b]
	\arrow/>>/[x`a;e]
	\arrow/>->/[e`a;m]
	\arrow/-->/[x`e;!\exists h]
\efig$$
obtaining a regular monomorphism $\mor{m}{\mathit{Eq}(a,b)}{Y}$. If $a \circ e = b \circ e$ then by the property of the equaliser, there exists a factorisation $e = m \circ h$; since $e$ is extremal, $m$ has to be an isomorphism, and so: $x = y$.



Finally, every regular epimorphism (i.e.~an epimorphism that appears as a coequaliser) is strong. Let us assume that a morphism $\mor{e}{X}{Y}$ coequalises a pair of morphisms $\mor{x,y}{Z}{X}$ and consider the diagram for orthogonality:
$$\bfig
	\node x(-400, 400)[Z]
	\node a(0, 400)[X]
	\node b(400, 400)[Y]
	
	\node c(0, 0)[A\;]
	\node d(400, 0)[B]

	\arrow/>>/[a`b;e]
	\arrow/-->/[b`c;\exists ! d]
	
	\arrow/>->/[c`d;m]
	\arrow[a`c;f]
	\arrow[b`d;g]

	\arrow/{@{>}@/^0.5em/}/[x`a;x]
	\arrow|b|/{@{>}@/_0.5em/}/[x`a;y]
\efig$$ 
Because $g \circ e = m \circ f$ and $\mor{e}{A}{B}$ coequalises $\mor{x,y}{Z}{X}$, we have that $m \circ f \circ x = m \circ f \circ y$. However, $m$ is a monomorphism, therefore $f \circ x = f \circ y$, and by the universal property of coequaliser $e$, there exists a unique morphism $\mor{d}{Y}{A}$ such that $d \circ e = f$. Because $e$ is an epimorphism the condition $m \circ d = g$ is equivalent to condition $m \circ d \circ e = g \circ e$, and by commutativity of the square to $m \circ d \circ e = m \circ f$, which is obviously true, and moreover, equivalent to the condition $d \circ e = f$.

We say that a regular (strong, extremal, epi) morphism $\mor{e}{X}{Y}$ is stable under pullbacks, if for every morphism $\mor{f}{Z}{Y}$ the pullback morphism $P \rightarrow Z$ is regular (resp.~strong, extremal, epi).

The next theorem is our main result.
\begin{theorem}[The other pullback lemma]\label{t:the:other:pullback:lemma}
Let $\mor{e}{X}{Y}$ be a morphism in a category with pullbacks\footnote{A careful reader may follow the proof presented in Section \ref{sec:proof} to see that we do not need \emph{all} pullbacks.}. The following conditions are equivalent:
\begin{itemize}
\item the other pullback lemma along $\mor{e}{X}{Y}$ holds 
\item $\mor{e}{X}{Y}$ is a strong morphism stable under pullbacks.
\item $\mor{e}{X}{Y}$ is an extremal morphism stable under pullbacks.
\end{itemize}
\end{theorem}

A category with finite limits is called \emph{regular} if every morphism $\mor{f}{X}{Y}$ factors as an extremal morphism $\mor{e}{X}{f[X]}$ followed by a monomorphism $\mor{j}{f[X]}{Y}$, and extremal morphisms are stable under pullbacks. Examples of regular categories include: elementary toposes, pretoposes, quasitoposes, abelian categories, the category of compact Hausdorff spaces, the category of groups, and more generally, any category monadic over $\catw{Set}$. On the other hand, Example~\ref{e:another} shows that the category of posets $\catw{Pos}$ is not regular, because the regular epimorphism $\mor{e}{2 \sqcup 2}{3}$ does not satisfy the other pullback lemma, therefore cannot be stable under pullbacks.

Because in a regular category every extremal morphism is a regular epimorphism (Lemma 4.4.6 in \cite{fib}, or Proposition 2.1.4 together with Proposition 2.2.2 in Vol.~2 \cite{borceux}), we may write the following corollary. 
\begin{corollary}
Let $e$ be a morphism in a regular category. The other pullback lemma holds along $e$ if and only if $e$ is a regular epimorphism.
\end{corollary}
The above corollary appears as Lemma 1.3.2 in \cite{elephant}, but it can also be obtained by a transfer principle between extremal epimorphisms of regular categories and epimorphisms of Grothendieck toposes (see Metatheorem 2.7.4 in \cite{borceux}). 

A category is called \emph{balanced} if every morphism that is both a monomorphism and an epimorphism is an isomorphism. A more sophisticated description of a balanced category is the following: every epimorphism is extremal. Clearly, if every epimorphism is extremal, then the category is balanced. In the other direction, let us assume that a category is balanced, and an epimorphism $e$ factors as $m \circ h$, where $m$ is a monomorphisms. By the definition of an epimorphism, $m$ is also an epimorphism, and since the category is balanced, it is an isomorphism. Therefore, $e$ is extremal.

\begin{corollary}
Let $e$ be a morphism in a balanced regular category. The other pullback lemma holds along $e$ if and only if $e$ is an epimorphism.
\end{corollary}
Every topos is balanced, because every monomorphism in a topos is regular, as a pullback of a regular monomorphism\footnote{Recall that a regular monomorphism is a monomorphism that appears as an equaliser; since limits are stable under pullbacks it follows that regular monomorphisms are stable under pullbacks.} $\mor{\top}{1}{\Omega}$ along the characteristic map. More generally, every pretopos (but not quasitopos!) is balanced, because in a pretopos every monomorphism equalises its cokernel pair\footnote{A cokernel pair of a morphism $\mor{f}{A}{B}$ is the pair of morphisms $B \rightarrow C$ obtained by forming the pushout $C$ of $f$ along itself.} (Proposition 5.8.10 in \cite{paul}).

\section{Another characterisation}
\label{sec:characterisation}
The aim of this section is to reformulate pullback lemmata in a more natural geometric language. Let us first recall some basic definitions \cite{borceux}\cite{fib}.

\begin{definition}[Cartesian morphism]\label{d:cartesian}
Let $\mor{p}{\catl{C}}{\catl{B}}$ be a functor. A morphism  $\mor{f}{B}{C}$ in $\catl{C}$ is $p$-cartesian (or just cartesian), if for every morphism $\mor{g}{A}{C}$ in $\catl{C}$ and every decomposition $p(g) = p(f) \circ k$ in $\catl{B}$ there is exactly one morphism $\mor{h}{A}{B}$ such that $p(h) = k$ and $g = f \circ h$:
$$\bfig
   \node c(-400, 600)[\catl{C}]
	\node b(-400, 300)[\catl{B}]
	\arrow[c`b;p]

	\node x(0, 900)[A]
	\node y(400, 600)[B]
	\node z(1000, 600)[C]

	\arrow|l|/-->/[x`y;h]
   \arrow|l|/->/[y`z;f]
   \arrow/{@{>}@/^1em/}/[x`z;g]

   \node i(0, 300)[p(A)]
	\node j(400, 0)[p(B)]
	\node k(1000, 0)[p(C)]

	\arrow|l|[i`j;k]
   \arrow|l|[j`k;p(f)]
   \arrow/{@{>}@/^1em/}/[i`k;p(g)]
\efig$$
\end{definition}
From the universal property of cartesian morphisms it follows that cartesian morphisms compose.
\begin{corollary}\label{c:cart:compose}
Let $\mor{p}{\catl{C}}{\catl{B}}$ be a functor. If both morphisms $\mor{f}{A}{B}$ and $\mor{g}{B}{C}$ are $p$-cartesian, then their composition $\mor{g \circ f}{A}{C}$ is cartesian as well.
\end{corollary}
Furthermore, cartesian morphisms are stable under factorisations.
\begin{corollary}\label{c:cart:factors}
Let $\mor{p}{\catl{C}}{\catl{B}}$ be a functor. If a $p$-cartesian morphism ${\mor{f}{A}{B}}$ factors as a morphism ${\mor{c}{A}{C}}$ followed by a $p$-cartesian morphism ${\mor{g}{C}{B}}$, then ${\mor{c}{A}{C}}$ is also a $p$-cartesian morphism.
\end{corollary}
\begin{proof}
Let $\mor{s}{X}{C}$ be any morphism in $\catl{C}$ and $\mor{k}{p(X)}{p(A)}$ be a morphism in $\catl{B}$ satisfying $p(c) \circ k = p(s)$. By the universal property of cartesian morphism $\mor{f}{A}{B}$ there is a unique $\mor{h}{X}{A}$ over $\mor{k}{p(X)}{p(A)}$ such that $f \circ h = g \circ s$. Observe that $c \circ h = s$, because both $c \circ h$ and $s$ are factorisations of $f \circ h$ through cartesian morphism $g$. Finally, $h$ is unique with this property, since if $s$ factors through $k$ as any $\mor{h}{X}{A}$ then by composition it also factors through the cartesian morphism $f$.  
\end{proof}
If we denote by $\arrl{B}$ the arrow category over a category $\catl{B}$, than we call the codomain projection functor:
$$\mor{\word{cod}(\catl{B})}{\arrl{B}}{\catl{B}}$$
the fundamental indexing of $\catl{B}$. The point is that cartesian morphism in the fundamental indexing of $\catl{B}$ are precisely the pullbacks in $\catl{B}$. Therefore, the pullback lemmata  may by abstractly rephrased as Corollary \ref{c:cart:compose} and Corollary \ref{c:cart:factors}. In order to characterise the other pullback lemma, we have to introduce a new concept.
\begin{definition}[Conservatively cartesian morphism]\label{d:conservative:cartesian}
Let $\mor{p}{\catl{C}}{\catl{B}}$ be a functor. A morphism $\mor{f}{A}{B} \in \catl{C}$ is $p$-conservatively cartesian (or just conservatively cartesian) if for every morphism $\mor{g}{B}{C}$ the following holds: $f \circ g$ is cartesian iff $g$ is cartesian.
\end{definition}
We will be interested in morphisms $\mor{e}{X}{Y}$ such that there are ``enough'' cartesian morphisms $f$ that lie over $e$ (i.e.~$p(f) = e$), and such that every cartesian morphism that lies over $e$ is conservatively cartesian. The next definitions will be useful.
\begin{definition}[Cartesian liftings]\label{d:lifting}
Let $\mor{p}{\catl{C}}{\catl{B}}$ be a functor. For an object $B \in \catl{C}$ and a morphism $\mor{j}{X}{p(B)} \in \catl{B}$, we say that a morphism $\mor{\lift{B}{j}}{A}{B}$ is a cartesian lifting of $j$ if it is cartesian and it lies over $j$: i.e.~ $p(\lift{B}{j}) = j$, like on the following diagram:
$$\bfig
   \node c(0, 300)[\catl{C}]
	\node b(0, 0)[\catl{B}]
	\arrow[c`b;p]

	\node y(300, 300)[A]
	\node z(900, 300)[B]

	 \arrow/-->/[y`z;\lift{B}{j}]
  
   \node j(300, 0)[X]
	\node k(900, 0)[p(B)]
	
   \arrow[j`k;j]
\efig$$
Furthermore, we say that $\mor{j}{X}{Y}$ has cartesian liftings, if for every $B \in \catl{C}$ such that $p(B) = Y$ there is a cartesian lifting of $j$.
\end{definition}

\begin{definition}[Fibration]\label{d:fibration}
A functor $\mor{p}{\catl{C}}{\catl{B}}$ is a fibration if every morphism $\mor{j}{X}{Y}$ has cartesian liftings.
\end{definition}
At this point, let us also recall that a morphism $\mor{j}{X}{Y}$ that has cartesian liftings induces a functor from the fibre over $Y$ to the fibre over $X$. More formally, for $Z \in \catl{B}$ let $p_Z$ denote the subcategory of $\catl{C}$ consisting of objects $A$ such that $p(A) = Z$ together with vertical morphisms over $A$ --- morphisms $\mor{f}{A}{B}$ such that $p(f) = \id{Z}$. If we assume a suitable variant of the Axiom of Choice in our meta set theory, we may associate with every object $A \in p_Y$ a cartesian morphism ${\mor{\lift{A}{j}}{j^*(A)}{A}}$. Then, there is a functor $\mor{j^*}{p_Y}{p_X}$ defined on objects as
$$A \mapsto j^*(A)$$
and on morphisms $A \to^{f} B$ as the unique factorisations of ${f \circ \lift{A}{j}}$ through cartesian morphism $\lift{B}{j}$:
$$\bfig
	\node c(-400, 200)[\catl{C}]
	\node b(-400, -100)[\catl{B}]
	\arrow[c`b;p]

	\node x1(0, 600)[j^*(A)]
	\node x2(0, 200)[j^*(B)]
	\node a(600, 600)[A]
	\node b(600, 200)[B]
	\arrow[a`b;f]
	\arrow/-->/[x1`x2;j^*(f)]
	\arrow[x1`a;\lift{A}{j}]
	\arrow[x2`b;\lift{B}{j}]

	\node j(0, -100)[X]
	\node k(600, -100)[Y]
	\arrow[j`k;j]
\efig$$
We call such defined functor $\mor{j^*}{p_Y}{p_X}$ the ``reindexing functor along $j$''. 

\begin{definition}[$p$-conservative morphism]\label{d:cartesian:extremal}
Let $\mor{p}{\catl{C}}{\catl{B}}$ be a functor. Assume that a morphism $\mor{j}{X}{Y} \in \catl{B}$ has cartesian liftings in $p$. We call $j$ a $p$-conservative morphism if its every lifting in $p$ is a $p$-conservatively cartesian morphism.
\end{definition}
From above definitions it follows, that $\word{cod}(\catl{B})$-conservative morphisms are precisely the morphisms that satisfy the other pullback lemma.

\section{Necessary and sufficient conditions}
\label{sec:proof}
A functor is conservative if it reflects isomorphisms. Let us observe that for reindexing functors, we can reformulate conservativity in a choice-free manner. If $\mor{p}{\catl{C}}{\catl{B}}$ is a functor and a morphism $\mor{e}{X}{Y} \in \catl{B}$ has cartesian liftings, then the reindexing functor $\mor{e^*}{p_Y}{p_X}$ is conservative iff for any pair of cartesian morphisms over $e$ with common codomain $\mor{f}{A}{B}$, $\mor{g}{A}{B'}$ if there is a vertical morphism $\mor{h}{B}{B'}$ such that $h \circ f = g$, then it is an isomorphism:
$$\bfig
	\node c(-400, 200)[\catl{C}]
	\node b(-400, -100)[\catl{B}]
	\arrow[c`b;p]

	\node x(0, 350)[A]
	\node a(600, 500)[B]
	\node b(600, 200)[B']
	\arrow[x`a;f]
	\arrow[x`b;g]
	\arrow[a`b;h]

	\node j(0, -100)[X]
	\node k(600, -100)[Y]
	\arrow[j`k;e]
\efig$$
In one direction it is obvious, because all reindexing functors over $e$ are naturally isomorphic to each other, therefore if one is conservative, then all are conservative, and from the above diagram there exists a reindexing functor which reindexes $\mor{h}{B}{B'}$ to the identity over $A$, thus $h$ is an isomorphism. In the other direction it follows from the definition of the reindexing functor. We leave it to a careful reader to reformulate the below two lemmata in a choice-free form.

\begin{lemma}
Let $\mor{p}{\catl{C}}{\catl{B}}$ be a fibration. The reindexing functor $\mor{e^*}{p_Y}{p_X}$ is conservative iff $\mor{e}{X}{Y}$ is $p$-conservative.
\end{lemma}
\begin{proof}
First let us assume that $\mor{e^*}{p_Y}{p_X}$ is conservative. Consider a diagram in $\catl{C}$:
$$\bfig
	\node a(0, 0)[A]
	\node b(1200, 0)[B]
	\node c(1650, 0)[C]
	\arrow[a`b;f]
	\arrow[b`c;c]

	\node k(1200, 350)[p(c)^*(C)]
	\arrow[b`k;!\exists h]
	\arrow[k`c;\lift{C}{p(c)}]
	
	\node r(0, 350)[p(f)^*(p(c)^*(C))]
	\arrow[r`k;\lift{p(c)^*(C)}{p(f)}]
	\arrow[a`r;p(c)^*(h)]
\efig$$
where both $\mor{f}{A}{B}$ and $\mor{c \circ f}{A}{C}$ are cartesian, $\mor{c}{B}{C}$ is any morphism, and $\mor{h}{B}{p(c)^*(C)}$ is defined as the unique morphism from $c$ to the cartesian morphism $\lift{C}{p(c)}$. Because the square and the triangle commute, the whole diagram is commutative as well. Therefore, $p(c)^*(h)$ is the unique factorisation of cartesian $c \circ f$ through cartesian $\lift{C}{p(c)} \circ \lift{p(c)^*(C)}{p(f)}$, and so it is an isomorphism. Thus, by conservativity of reindexings, also $h$ is an isomorphism. Which means that $c$ is cartesian as the composition of a cartesian morphism with an isomorphism.

The other direction of the lemma is obvious.
\end{proof}
We say that for a functor $\mor{p}{\catl{C}}{\catl{B}}$, a morphism $\mor{e}{X}{Y} \in \catl{B}$ has op-cartesian liftings if $\mor{e^{op}}{Y}{X} \in \catl{B}^{op}$ has cartesian liftings in the opposite functor $\mor{p^{op}}{\catl{C}^{op}}{\catl{B}^{op}}$. For an object $A \in p_X$ we denote the op-cartesian lifting over $e$ by $\mor{\oplift{A}{e}}{A}{e_!(A)}$, and the associated op-reindexing functor by $\mor{e_!}{p_X}{p_Y}$. By universal properites of (op-)cartesian morphisms, a morphism $\mor{e}{X}{Y} \in \catl{B}$ that has cartesian liftings, also has op-cartesian liftings iff the reindexing functor $\mor{e^*}{p_Y}{p_X}$ has left adjoint $\mor{e_!}{p_X}{p_Y}$:
$$\hom_X(A, e^*(B)) \approx \hom_e(A, B) \approx \hom_Y(e_!(A), B)$$
where $\hom_e(A, B)$ consists of these morphisms from $\catl{C}$ that lay over $e$.

There is one well-known fact relating right adjoint functors and epimorphisms. Let $\mor{F}{\catl{D}}{\catl{C}} \vdash \mor{G}{\catl{C}}{\catl{D}}$ be an adjunction with $G$ right adjoint to $F$. Consider counit $\mor{\epsilon_A}{F(G(A))}{A}$ of the adjunction at $A$ and two morphisms $\mor{f,g}{X}{G(F(A))}$ such that: $f \circ \epsilon_A = g \circ \epsilon_B$. By transposition, we obtain morphism: $$(G(f) = G(g)) = (f \circ \epsilon_A = g \circ \epsilon_B)^\dag \colon G(X) \rightarrow G(A)$$
Therefore $G$ is faithful (i.e.~injective on morphisms) iff every counit is an epimorphism.
 
There is also another well-known ``fact'', which is widely used across research papers, that relates adjoint functors and strong epimorphisms --- a right adjoint is conservative iff the counits of the adjunction are strong epimorphisms. Unfortunately, the ``fact'' is all wrong as is stated and one may easily (yet not without much work) construct a counterexample (see Example 2.4 in \cite{counter}). The right version of the ``fact'' was first studied by John Isbell. Here we present the result described in \cite{kelly}.

\begin{lemma}\label{l:cons:extremal}
Let $\mor{p}{\catl{C}}{\catl{B}}$ be a functor and assume that a morphism ${\mor{e}{X}{Y}}$ form $\catl{B}$ has cartesian liftings and op-cartsian liftings. If the reindexing functor $\mor{e^*}{p_Y}{p_X}$ is conservative then the canonical morphisms ${e_!(e^*(A)) \rightarrow A}$ in $p_Y$ are extremal.
\end{lemma}
\begin{proof}
Let us assume that $\mor{e^*}{p_Y}{p_X}$ is conservative. We have to show that the counits $\mor{\epsilon_A}{e_!(e^*(A))}{A}$ of the adjunction are extremal. Let ${\epsilon_A = m \circ f}$ be any factorisation with $m$ being a monomorphism. Then because right adjoints preserve monomorphisms, $e^*(\epsilon_A) = e^*(m) \circ e^*(f)$ is a factorisation of $e^*(\epsilon_A)$ through subobject $e^*(m)$. By the triangle equality for the counit $e^*(\epsilon_A)$ is split by the unit $\eta_{e^*(A)}$ and so $e^*(m)$ is split by $e^*(f) \circ \eta_{e^*(A)}$. Therefore $e^*(m)$ is an isomorphism and by conservativity of $e^*$, $m$ is also an isomorphism.
\end{proof}

On condition the fibre $p_Y$ has equalisers, the following lemma is the converse of Lemma \ref{l:cons:extremal}.
\begin{lemma}\label{l:extremal:cons}
Let $\mor{p}{\catl{C}}{\catl{B}}$ be a functor and assume that a morphism ${\mor{e}{X}{Y}}$ form $\catl{B}$ has cartesian liftings and op-cartsian liftings. If the canonical morphisms ${e_!(e^*(A)) \rightarrow A}$ in $p_Y$ are extremal epimorphisms, then the reindexing functor $\mor{e^*}{p_Y}{p_X}$ is conservative.
\end{lemma}
\begin{proof}
Let us assume that each counit $\mor{\epsilon_A}{e_!(e^*(A))}{A}$ is extremal epimorphism. We have to show that for every $\mor{f}{A}{B}$ if $e^*(f)$ is an isomorphism then $f$ is an isomorphism. Since any functor preserves isomorphisms if $e^*(f)$ is iso, then $e_!(e^*(f))$ is iso. By definition of the adjunction $f \circ \epsilon_A = \epsilon_B \circ F(G(f))$, so $f \circ \epsilon_A \circ F(G(f))^{-1} = \epsilon_B$ and thus $f$ is extremal. Indeed, if $f = m \circ h$, then $\epsilon_B = m \circ h \circ \epsilon_A \circ F(G(f))^{-1}$ so if $m$ is mono, $m$ has to be iso. On the other hand, each counit is also an epimorphism, so the right adjoint $e^*$ reflects monomorphisms. Therefore, if $e^*(f)$ is an isomorphism, then $f$ is both extremal and monomorphism, which means, it is an isomorphism.
\end{proof}



\begin{corollary}
Theorem \ref{t:the:other:pullback:lemma} holds.
\end{corollary}
\begin{proof}
In the fundamental indexing $\mor{\word{cod}(\catl{B})}{\arrl{B}}{\catl{B}}$, every morphism ${\mor{e}{X}{Y}}$ has op-cartesian liftings given by compositions. Therefore, if we assume that all pullbacks along $e$ exist, the reindexing functor $\catl{B}/Y \rightarrow \catl{B}/X$ exists and has left adjoint. Moreover, morphisms in $\catl{B}/Y$ are extremal precisely when they are extremal in $\catl{B}$, and since $\catl{B}/Y$ has equalisers, they are extremal epimorphisms.  
\end{proof}

	

	




\begin{thebibliography}{}
%
%
\bibitem{maclane}
S.~Mac Lane, \textit{Categories for the Working Mathematician}, Springer Verlag (1997)

\bibitem{elephant} P. T. Johnstone, \textit{Sketches of an Elephant: A Topos Theory Compendium}, Oxford University Press (2003)

\bibitem{borceux}
F.~Borceux, \textit{The Handbook of Categorical Algebra}, Cambridge University Press (1994)

\bibitem{fib}
B.~Jacobs, \textit{Categorical Logic and Type Theory}, Elsevier (2001)

\bibitem{paul}
P.~Taylor, \textit{Practical Foundations of Mathematics}, Cambridge University Press (1999)

\bibitem{counter}
R.~Borger, W.~Tholen, \textit{Strong regular and dense generators. Cahiers de Topologie et Geometrie Différentielle Categoriques}, 32(3) (1991)

\bibitem{kelly}
G.~Bin~Im and G.M.~Kelly, \textit{Some remarks on conservative functors with left adjoints}, Journal of the Korean Mathematical Society 23 (1986)

\end{thebibliography}


\end{document}